\documentclass[10pt]{article}
\usepackage{graphicx}
\usepackage{latexsym}
\usepackage{float}
\usepackage{amsmath}
\usepackage{cite}
\usepackage{amssymb}
\usepackage{amsthm}
\usepackage{url}
\usepackage{tabulary}
\usepackage{booktabs}
 \usepackage{helvet}

\DeclareMathOperator{\tdeg}{tdeg}
 
\theoremstyle{plain}
\newtheorem{theorem}{Theorem}[section]

\newtheorem{lemma}[theorem]{Lemma}
\newtheorem{Ex}[theorem]{Example}

\theoremstyle{definition}
\newtheorem{definition}{Definition}[section]

\theoremstyle{remark}

\def \lmk {l_{\mathbf{m},k}}

\def \W {{\mathbb W}}
\def \Z {{\mathbb Z}}

\def \K {{\mathbb K}}

\def \u {{\bf u}}
\def \0 {{\mathbf 0}}
\def \a {{\underline  a}}

\def \x {{\underline  x}}
\def \u {{\underline  u}}
\def \K {{\mathbb K}}
\def \m {{\mathbf m}}
\def \n {{\mathbf n}}
\def \0 {{\mathbf 0}}
\def \S {{\underline S}}
 
\begin{document}

\title{Irreducibility of integer-valued polynomials in several variables 
}


\author{Devendra Prasad 
}


  \author{Devendra Prasad\\
   devendraprasad@iisertirupati.ac.in\\
   Department of Mathematics\\ IISER-Tirupati, Tirupati, Andhra Pradesh\\
 India, 517507\\  
  }

\date{Received: date / Accepted: date}

\maketitle

\begin{abstract}  Let $\S $ be an arbitrary subset of $R^n$  where $R$ is  a domain with the field of fractions $\K$. Denote the ring of   polynomials in $n$ variables over $\K$ by $\K[\x].$  The ring of integer-valued polynomials over $\S ,$ denoted by Int$(\S,R)$, is defined as the set of the  polynomials  of $\K[\x],$ 
 which maps $\S$ to $R$. In this article, we study the irreducibility of the polynomials of Int$(\S,R)$ for the first time in the case when $R$ is a Unique Factorization Domain.  We also show that our results remain valid when $R$ is a Dedekind domain or sometimes any domain.
 
\end{abstract}
 \textbf{keywords:} Integer-valued polynomials, irreducibility

 
\section{Introduction} 
 
  Let $R$ be    a domain with the field of fractions $\K$.  
 For a given subset $\S \subseteq R^n,$ where $n \geq 1,$  consider the following subset of $\K[\x]\ (=\K[x_1, \ldots, x_n])$
 $$\mathrm{Int}(\S,R) =\{ f \in \K[\x]: f(\S) \subseteq R \}.$$
 It can be verified very easily that this set forms a   ring  and is known as the ring of integer-valued polynomials over $\S.$ In the last few decades, this ring has been a center of attraction for commutative algebraists.  This ring is widely used to construct examples/counterexamples in commutative algebra.  We refer to Cahen and Chabert \cite{Cahen} for a general reference.
 
 \medskip 
 
 In the study of ring theory, one of the most exciting concepts is irreducibility. In factorization theory also, it is very crucial to check the factors of a given polynomial. Hence, we must be familiar with the  irreducibility  of a given   polynomial.  Recently, Prasad \cite{prasadirrI} gave a new approach to test the irreducibility of a given integer-valued polynomial in one variable. The cornerstone of this study was the construction of $\pi$-sequences and  $d$-sequences, which we recall here for the sake of completeness.

     \begin{definition} A sequence $\{ u_i \}_{i \geq 0}$ of elements of  $S \subset R$, where $R$ is a unique factorization domain, is said to be a $\pi$-sequence if   for each   $k>0$,    $u_k \in S$   satisfies  
 
 $$\tfrac{(x-u_0)\ldots (x- u_{k-1})}{(u_k-u_0) \ldots (u_k-u_{k-1})} \in \mathrm{Int}(S,R_{(\pi)} ).$$

\end{definition} 

  The $d$-sequences are defined as follows.
    
    \begin{definition} For a given element $d \in R$, where $R$ is a unique factorization domain,
     let $\pi_1, \pi_2,\ldots, \pi_r$
 be all the   irreducibles   of $R$ dividing  $d$. Let for   $1 \leq j \leq r$, $\{ u_{ij} \}_{i \geq 0} $ be a $\pi_j$-sequence of $S$ and $\pi_j^{e_{kj} }$ be   $(u_{kj}-u_{0j})\ldots (u_{kj}- u_{k-1j})$ viewed as  a member of  the ring $R_{(\pi_j)}$. Then a $d$-sequence $ \{ x_i \}_{0 \leq i \leq k}$   of $S$  of length $k$  is a solution to the following congruences

   \begin{equation}\label{crt}
   x_i \equiv u_{ij} \mod \pi_j^{e_{kj}+1}\ \forall\ 1 \leq j \leq r,
   \end{equation}

  where  $0 \leq i \leq k.$ 
\end{definition}
   
   The following criterion was obtained for testing the irreducibility of an integer-valued polynomial.

  \begin{theorem}\label{mainthpre}  Let $S$ be an arbitrary subset of a unique factorization domain $R$. Assume $f = \tfrac{g}{d} \in \mathrm{Int}(S,R  ) $ is  a polynomial  of degree $k$ and $a_0, a_1, \ldots, a_k$ be a $d$-sequence. Then $f$ is irreducible in $\mathrm{Int}(S,R  ) $ iff the following holds:

for any factorization $g=g_1g_2$ in $R[x]$   and    a  divisor $\pi$  of $d$ such that $e_k$ is the maximum integer satisfying  $\pi^{e_k} \mid g_1(a_i)\ \forall\ 0 \leq i \leq \deg(g_1) , $   there exists an integer $j$  satisfying  $ 0 \leq j \leq \deg(g_2)  $ and    ${w_{\pi}(\tfrac{d}{\pi^{e_k}})}   \nmid  g_2(a_j).$
  
  \end{theorem}
      \medskip

In this way,  a criterion was   obtained for  the first time  to check the  irreducibility of a given integer-valued polynomial in a very general setting.  Observe that $d$-sequences can be obtained for any subset of a Dedekind domain (or sometimes  any domain). As a consequence, Theorem \ref{mainthpre}  remains viable for all   domains where $d$-sequences can be constructed.   It is natural to think about a generalization of Theorem \ref{mainthpre}  to the multivariate case.
 
 \medskip
 
 The irreducibility of a multivariate polynomial has a venerable history of two centuries. Several mathematicians have studied the irreducibility of a multivariate polynomial. However, as per our knowledge, this topic is never explored for the ring $\mathrm{Int}(\S,R)$ even in the simplest case when $\S=\Z^n,$ 
 where $n >1$ and $R =\Z.$ In this article, we study the irreducibility  of a multivariate integer-valued polynomial for the first time.
 
 \medskip

 The summary of the article is as follows. In Section \ref{SecPrenotassu}, we fix some notations for the whole article and recall some known/unknown concepts. 
 We introduce the notion of $d_{\m}$-sequences with some examples in Section  \ref{secd}. In Section  \ref{secirr}, we make $d_{\m}$-sequences   our main tool of study to prove our main theorem for testing the irreducibility of a given polynomial.  Section \ref{secirrgen} presents a generalization of our results in some special cases.

 \section{Preliminaries, notations and assumptions}\label{SecPrenotassu}

 In this section, we recall some known/unknown concepts and  fix some notations for the whole article. Throughout the article  $\W$ denotes the set $\{0,1,2, \ldots \}.$  
For a given $n \in \W$,   $\underline{S}$ denotes  an arbitrary subset of $R^n$ where $R$ is a unique factorization domain   with the field of fractions $\K$. For a given polynomial $f \in \K[\underline{x}] (=\K[x_1,x_2, \ldots, x_n]),$      $\tdeg (f)$ denotes the total degree of $f$. We call $ \deg (f) $ is equal to $\m =(m_1,\ldots, m_n) \in \W^n,$ if the degree of $f$ in the  $i$th variable is $m_i\ \forall\ 1 \leq i \leq n$. A polynomial $f \in \K[\underline{x}] $ is said to be  of type $(\m,k)$ if $\deg (f)= \m$ and $\tdeg(f)=k$. For given $\m$ and $\n$ in $\W^n$, we say $\m \leq  \n$ if each component of $\m$ is less than or equal to the corresponding component of $\n$.

\medskip

Take a unitary monomial basis of  
$\K [\underline{x}]$  and place a total order on it,  which 
is compatible with the total degree. 
We fix this total order    once and for all as it is very important throughout the study. In this way, the monomials are arranged in a sequence $(p_j)_{j \geq 0
}$  with  $p_0=1$  and $\tdeg(p_i) \leq \tdeg(p_j)$ if $i<j$. 

\medskip

Let $\K_{(\m,k)}[\x]$ denote the vector space of all type $ (\m,k)$ polynomials of $\K[\underline{x}].$ 
We denote by $\lmk$ the cardinality of a basis of $\K_{(\m,k)}[\x].$            Hence, the first $\lmk $ terms of the monomial ordering fixed above are sufficient  to express any   polynomial $f \in\K_{(\m,k)}[\x].$ If each component of $\m \in \W^n$ tends to infinity, then $\m$ does  not `remove' any of the polynomials in the basis of $\K[\x]$. In such a case, we say  $\m$ is {\em  sufficiently large.} Also, for a given polynomial  $f \in \K[\underline{x}] $  of type $(\m,k),\ l(f)$ denotes the number $\lmk$.

\medskip

 For  a fixed $\m \in \W^n$ and a given sequence of elements $\underline{ a}_0, \underline{a}_1,\ldots, \underline{a}_{r}$ in $R^n$  recall that

\begin{equation*}\label{definition of new delta}
 \Delta_{\m}(\underline{a}_0, \underline{ a}_1, \underline{a}_2,\ldots, 
\underline{a}_{r})= \det(p_j(\underline{a_i}))_{0 \leq i, j \leq r}.
\end{equation*}

\medskip

These determinants were also studied in Rajkumar, Reddy and Semwal \cite{Devendra} (see also \cite{prasadsurvey} and \cite{ev}) in the case of a Discrete Valuation Ring. For a given polynomial  $f= \tfrac{g}{d} \in \K[\x],$ we assume that $ g$ is the unique polynomial in $R[\x]$ and $d$ is also unique in $R$.

\medskip

A polynomial $f \in \K[\x]$ is said to be irreducible over $\K[\x]$ if it cannot be factored as $g_1(\x)g_2(\x)$ where both  of the polynomials 
have the degree greater than $\mathbf{0},$ where $\mathbf{0} \in \W^n$ is the vector with each component equal to zero. A polynomial  $f   \in \mathrm{Int}(\S,R  )   $ is said to be an  {\em `image primitive' } polynomial if we cannot find an irreducible element $\pi \in R$ such that   $\pi \mid f(a)\ \forall\ a \in \S .$ If a polynomial is not image primitive then there exists a $d \in R$ such that $d \mid f(a)\ \forall\ a \in \S .$ Now have the following factorization 

$$f =d \times \tfrac{f}{d}    $$

 in the ring $\mathrm{Int}(\S,R  ).$ Consequently, the given polynomial $f$ is reducible in the ring $\mathrm{Int}(\S,R  )   .$ Hence, we always assume that a given polynomial is image primitive.
   For brevity, by an integer-valued  polynomial,  we mean a    polynomial in Int$(\S,R)$, where $\S$ and $R$  automatically come from the context.  For an irreducible element $\pi$ and a given element $d \in R$, $w_{\pi}(d)$ denotes the highest power of  $\pi$ dividing $d$. For instance, $w_3(18)=3^2.$

  \section{$d_{\m}$-sequences} \label{secd}

 In Prasad \cite{prasadirrI} , the concept of $d$-sequences was introduced for the first time to check the irreducibility of a given integer-valued polynomial. In this section, we generalize the definition of $d$-sequences to the case of several variables.  We construct    special kinds of sequences called  {\em $d_{\m}$-sequences}. Before introducing these kinds of sequences, we need a generalization of {\em  $\pi$-sequences} called  $\pi_{\m}$-sequences, to the case of $n$ variables.  We know that a subset $\S \subseteq R^n$ is also a subset of  $R_{(\pi)}^n$ for every prime ideal $(\pi) \subseteq R$, where $R_{(\pi)}$ denotes the localization of $R$ at the prime ideal $(\pi)$. With this assumption, we give the following generalization of $\pi$-sequences.

     \begin{definition} For a fixed $\m \in \W^n,$ a sequence $\{\underline{u}_i \}_{i \geq 0}$ of elements of  $\underline{S} \subset R^n$ is said to be a $\pi_{\m}$-sequence if   for each   $k>0$,    $\u_k \in S$   satisfies  
 
 $$\tfrac{   \Delta_{\m}  (\u_0, \u_1,\ldots, \u_{k-1}, \underline{x} )}{\Delta_{\m}  (\u_0, \u_1,\ldots,  \underline{u}_k ))} \in \mathrm{Int}(\S,R_{(\pi)} ).$$


\end{definition}

   We take an assumption that in  a  $\pi_{\m}$-sequence $\{ \u_i \}_{i \geq 0},$ the first element $\u_0$ is selected arbitrarily. 
   Sequences similar to this were also studied by Rajkumar, Reddy and Semwal \cite{Devendra} in the case of a Dedekind domain to study   generalized factorials. Now we define  $d_{\m}$-sequences  as follows.

 \medskip


    \begin{definition} For a given element $d \in R$ and fixed ${\m} \in \W^n$,   
     let $\pi_1, \pi_2,\ldots, \pi_r$
 be all the   irreducibles   of $R$ dividing  $d$. Let for   $1 \leq j \leq r$, $\{ \u_{ij} \}_{i \geq 0} $ be a $\pi_{\m,j}$-sequence of $\S$ and $\pi_j^{e_{kj} }$ be   $\Delta_{\m} (\u_0, \u_1,\ldots,  \underline{u}_k ))$ viewed as  a member of  the ring $R_{(\pi_j)}$. Then a $d_{\m}$-sequence $ \{ \x_i \}_{0 \leq i \leq k}$   of $\S$  of length $k$  is a solution to the following congruences
 
   
   \begin{equation}\label{crt}
   \x_i \equiv \u_{ij} \mod \pi_j^{e_{kj}+1}\ \forall\ 1 \leq j \leq r,
   \end{equation}

  where  $0 \leq i \leq k.$ 
\end{definition}
      
  \medskip
   Throughout the article, we fix a solution of Eq. (\ref{crt}) for each $i$   and get a sequence $\a_0, \a_1, \ldots, \a_k$ of $k+1$ elements. 
    Such a sequence need not to be inside $\S.$ Clearly, a  $d_{\m}$-sequence also depends on the set chosen. If  the subset $\S$ is clear from the context,  we  call only a ` $d_{\m}$-sequence'  without mentioning the set.   This sequence is a cornerstone of  our study. Before proceeding, we give a few examples of $d_{\m}$-sequences.
  
\medskip
   In all the examples of $d_{\m}$-sequences  below, we take only two variables. We fix the following total ordering on the set of unitary monomial basis of $R[x,y]$ 
 $$1,x,y,x^2,xy,y^2,x^3,x^2y,xy^2,y^3, \ldots.$$

We assume that $x$ is always the first variable. $i.e$ if we say a polynomial is of type $((3,4),5)),$ then it means that the partial degree of $f$ in $x$ is three.

  \begin{Ex}  In the case when $R=\Z,\ \m$ is sufficiently large and $\S = \Z \times \Z,$  the sequence $$(0,0),(1,0),(0,1),(2,0),(1,1),(0,2), (3,0),(2,1),(1,2),\ \mathrm{and}\  (0,3)  $$ is  a $d_{\m}$-sequence of length $9$ for every $d \in \Z$. 

  \end{Ex} 
  
  We give an example of a $d_{\m}$-sequence when the given subset is finite.
  
  \begin{Ex} Let $R=\Z,\ \m$ be sufficiently large and $\S   $ is the set
  $$\S =\{(0,0),(1,0), (1,4),(4,0), (1,1), (4,1),(9,0),(0,1),(0,4),    (0,9) \} $$
      Then  
  $$(0,0),(1,0),(0,1),(4,0),(1,1),(0,4),  (9,0),(4,1),(1,4),\ \mathrm{and}\  (0,9)  $$ 
    
  form a  $d_{\m}$-sequence of length $9$ for every $d \in \Z$. 
  
  \end{Ex}

  \medskip
  
  We give a more general example of  $d_{\m}$-sequences.   Recall that, for a given subset $\S \subset R^n,$ the {\em fixed divisor } of a polynomial $f \in R[\x]$ over $\S$ is the greatest common divisor of the values taken by $f$ over $\S$. This quantity is denoted by $d(\S,f).$ Thus,
  
  $$d(\S,f)=  \gcd \{ f(\a): \a \in \S    \} .$$
  
  For some interesting results on the topic, we refer to   Rajkumar, Reddy and Semwal \cite{Devendra}, The whole article is devoted to the study of fixed divisors of multivariate polynomials. A nice application of fixed divisors in one and several variables  can be found in Prasad \cite{Prasad2019}.

    \medskip

   A sequence of distinct elements $\{ \a_i \}_{i \geq 0}$ of $\S$ is said to be  a {\em  fixed divisor sequence} (see   \cite{prasadsurvey} or \cite{Devendrafixed} )) if  for every $k >0,\ \exists\ l_k \in \Z,$  such that for every polynomial $f$ of total degree $k$  $$d(\S,f)=(f(\a_0),f(\a_1), \ldots, f(\a_{l_k}))
,$$  
  and no proper subset of $\{\a_0,\a_1, \ldots, \a_{l_k} \}$ 
  determines the fixed divisor of all the total degree $k$ polynomials. 

  \begin{Ex} Let $\m$ be sufficiently large and $\S $ be a subset of $R^n$ with a fixed divisor sequence $\{ \a_i \}_{i \geq 0}$. Assume 
  $$d(\S,F_r)=(F_r(\a_r))\ \forall\   r \geq 0, $$
 where $F_r(\x)=  \Delta_{\m}  (\a_0, \a_1,\ldots, \a_{r-1}, \underline{x} ).$ 
  Then $ \a_0,\a_1, \ldots, \a_k$ is a $d_{\m}$-sequence of length $k$ for every $d \in R.$

  \end{Ex}
  
\medskip

  In all the   examples given so far, we assumed that each component of ${\m} \in \W^2$ is sufficiently large.   Now we give an example when each  component of ${\m} \in \W^2$ is small and   plays a role in the calculation.

  \begin{Ex}Let $\mathbf{m} =(2,2),\ R=\Z$ and $\S =\Z \times \Z. $ Then 
  $$(0,0),(1,0),(0,1),(2,0),(1,1),(0,2), (2,1),(1,2)\ \mathrm{and}\ (2,2).  $$ 
  is a   $d_{(2,2)}$-sequence of length eight. The ninth term of the  $d_{(2,2)}$-sequence  does not exist.
  \end{Ex}

  \medskip
  
  The readers can easily see the difference between the first   example and the above example in which $\m$ appears.

  \medskip

  \section{Irreducibility of integer-valued polynomials} \label{secirr}

  We start this section with the following lemma, which is useful in proving   the main result.  
  
  \begin{lemma}\label{ivpcr} Let $\underline{a}_0, \underline{a}_1, \ldots, \underline{a}_l$ be a $d$-sequence of length $l$ for some $d \in R$ and   $l \in   \W$. 
   Then, for any  polynomial $ f' = \tfrac{g'}{d'}$ where $d' \mid d$ and    $l(f') \leq l,$ the following holds
  
    $$f' \in \mathrm{Int}(\underline{S},R  )   \Leftrightarrow f'(\underline{a}_i) \in R\ \forall\ 0 \leq i \leq  l(f').$$
  \end{lemma}

  \begin{proof} For a given  $\pi \mid d,$ let $\underline{b}_0, \underline{ b}_1, \underline{b}_2,\ldots  $ be a   a $\pi$-sequence in $\underline{S}$.    Consider   the following representation of a given polynomial  $f' \in \mathrm{Int}(\underline{S},R  )  $

$$f' =\tfrac{g'}{d'} =\sum_{i=0}^{ l(f')} c_i  \tfrac{\Delta_{\mathbf{m}}(\underline{b}_0, \underline{ b}_1, \underline{b}_2,\ldots, \underline{b}_{i-1}, \underline{x}) }{\Delta_{\mathbf{m}}(\underline{b}_0, \underline{ b}_1, \underline{b}_2,\ldots, \underline{b}_{i})  }.$$  
 in the ring Int$(\S,R_{(\pi)}).$ Observe that 
 
  \begin{equation*}
\begin{split}  f'(\underline{a}_r) \in R\ \forall\ 0 \leq r \leq l(f') \Rightarrow 
&  \sum_{i=0}^{l(f')} c_i  \tfrac{\Delta_{\mathbf{m}}(\underline{b}_0, \underline{ b}_1, \underline{b}_2,\ldots, \underline{b}_{i-1}, \underline{a}_r) }{\Delta_{\mathbf{m}}(\underline{b}_0, \underline{ b}_1, \underline{b}_2,\ldots, \underline{b}_{i})  } \in R_{(\pi)}\ \forall\ 0 \leq r \leq l(f')\\
 \Rightarrow &\     c_r     \in R_{(\pi)}\ \forall\ 0 \leq r \leq l(f') \end{split}
   \end{equation*}
  With this observation, for any $\a \in \S$
      $$f'(\a)=\sum_{i=0}^{k'} c_i  \tfrac{\Delta_{\mathbf{m}}(\underline{b}_0, \underline{ b}_1, \underline{b}_2,\ldots, \underline{b}_{i-1}, \a) }{\Delta_{\mathbf{m}}(\underline{b}_0, \underline{ b}_1, \underline{b}_2,\ldots, \underline{b}_{i})  }   $$  is   a member of $R_{(\pi)}$. Hence $ f' \in \mathrm{Int}(\S,R_{(\pi)}).$ This can be shown for all the divisors of $d,$ completing the one part.  

  Conversely, let $ f' \in \mathrm{Int}(\S,R)$
  then using the congruence $\a_r \equiv \underline{b}_r \pmod{{\pi^{e_k}}}$, we have 
  
  $$ g'(\a_r) \equiv g'(\underline{b}_r) \pmod{{\pi^{e_k}}}.$$
  
  It follows that $f'(\a) \in R\ \forall\ 0 \leq i \leq \lmk' .$


  \end{proof}

  \medskip


  
 Now we prove the main theorem of this section.

  \begin{theorem}\label{mainth}  Let $f = \tfrac{g}{d} \in \mathrm{Int}(\underline{S},R  ) $ be  a polynomial  of  type $(\m,k)$ and $\underline{a}_0, \underline{a}_1, \ldots, \underline{a}_{\lmk-1}$ be a $d_{\m}$-sequence. Then $f$ is irreducible in $\mathrm{Int}(\underline{S},R  ) $  iff the following holds:

for any factorization $g=g_1g_2$ in $R[x]$   and    a  divisor $\pi$  of $d$ such that $e_k$ is the maximum integer satisfying  $\pi^{e_k} \mid g_1(\underline{a}_i)\ \forall\ 0 \leq i \leq l(g_1) , $   there exists an integer $j$  satisfying  $ 0 \leq j \leq l(g_2)  $ and    ${w_{\pi}(\tfrac{d}{\pi^{e_k}})}   \nmid  g_2(\underline{a}_j).$
  
  \end{theorem}

 \begin{proof} For a given  polynomial $f = \tfrac{g}{d} \in \mathrm{Int}(\underline{S},R  )   $, suppose for every factorization    $g=g_1g_2$  in $R[x]$ 
 there exist a divisor  $\pi$  of $d$ satisfying $\pi^{e_k} \mid g_1(\underline{a}_i)\ \forall\ 0 \leq i \leq l(g_1) $  and  $    w_{\pi}(\tfrac{d}{\pi^{e_k}})   \nmid  g_2(\underline{a}_j)$ for some non-negative integer $  j \leq \deg(g_2).$ Let us assume contrary that $f$ is reducible. Hence, there exists a factorization 
 
 $$f= \dfrac{h_1}{d_1}  \dfrac{h_2}{d_2}, $$
 
in $\mathrm{Int}(\underline{S},R  ) $ such that $  \tfrac{h_1}{d_1}$ and $ \tfrac{h_2}{d_2}$ are members of $\mathrm{Int}(\underline{S},R  ). $ If for a divisor $\pi$ of $d$, $ w_{\pi}(d_1)=\pi^{e_k}= w_{\pi}(d) ,$ then this is clearly a  contradiction   since $\pi^{0} \mid h_2(a)\ \forall\ a \in R. $ By a similar argument, $ w_{\pi}(d_1)$ cannot be $\pi^{0}.$  Hence we conclude that  $w_{\pi}(d_1)$ cannot be equal to $ w_{\pi}(d) .$
 It follows that there exists a positive integer   $     j \leq \deg(h_2)$such that  $   w_{\pi}(\tfrac{d}{\pi^{e_k}})     \nmid  h_2(\underline{a}_j)$. By Lemma \ref{ivpcr}, the polynomial  $\tfrac{h_2}{d_2} $ cannot be a member of $\mathrm{Int}(\S,R  ), $ which is   a contradiction. Hence, the polynomial $f$ must be irreducible  in $\mathrm{Int}(\S,R  ) $.

Conversely, let    $f = \tfrac{g}{d} \in \mathrm{Int}(\S,R  ) $ be irreducible in $\mathrm{Int}(\S,R  ) $ then for any factorization  $g=g_1g_2$ in $R[x]$ we can find suitable $d_1$ and $d_2$ such that 
 $$f= \dfrac{h_1}{d_1}  \dfrac{h_2}{d_2}, $$
 
where $  \tfrac{h_1}{d_1} \in \mathrm{Int}(S,R )$ and $    \tfrac{h_2}{d_2} \notin \mathrm{Int}(S,R )$. Now 
by Lemma \ref{ivpcr} there exists a divisor $\pi$  of $d_2$, such that $ w_{\pi}(d_2) \nmid h_2(\underline{a}_i)$ for some $ 0 \leq i \leq \deg(h_2) .$ However, $ w_{\pi}( \tfrac{d}{d_2} )$ divides $ h_1(\underline{a}_j)\ \forall\ 0 \leq j \leq \deg(h_1) $ which completes the proof.


 \end{proof}

    \medskip

    We give an example to illustrate our theorem. We assume that $x$ is the first variable.

    \begin{Ex} Suppose we want to  check the   irreducibility of the bivariate polynomial $$ f=\tfrac{1}{4}( 4x^2y^2+4x^2y+4xy^3-4xy^2+10xy+2x+y^4-3y^3+5y^2-3y+4)  $$ in $\mathrm{Int}(\Z^2,\Z  ). $    The only   way to factorise $f$ is the  following 
    
    $$ f=  \tfrac{1}{4}   (y^2-3y+2x+2xy+4)(y^2+2xy+1)  .$$

  Here   both  polynomials are of type $((1,2),2)$ and we know $$
   S=\{(0,0),(1,0),(0,1),(1,1),(0,2) \}    $$ 
  is a $4_{(2,1)}$-sequence  of length four.
  We need to  check the values   at these points only. Let $f_1= y^2-3y+2x+2xy+4,$ then one is the largest positive integer such that $2^1 \mid f_1(i)\ \forall\ i \in S.$ However, for the polynomial  $f_2=y^2+2xy+1$,   $2^{2-1}$ does not divide  $f_2(0,0)= 1$. Hence, the given  polynomial $f$ is irreducible  in $\mathrm{Int}(\Z^2,\Z  )  $  by Theorem \ref{mainth}.

      \end{Ex}

  \section{Further generalizations}\label{secirrgen}

In this section, we suggest  a generalization of   Theorem \ref{mainth} for some more general domains. If the ideal generated by a given element $d \in D$ in the domain $D$ factors uniquely as a product of prime ideals, then by a similar way we can   get a $d_{\m}$-sequence in this setting as well. 
For the sake of completeness, we state an analogue of our result in the case of a Dedekind domain, where all the ideals factor uniquely as a product of prime ideals. Here, for a given prime ideal $P \subset  D$ and a given element $d \in D$,  $w_P(d)$ denotes the highest power of the prime ideal $P$ dividing $d$.

 \begin{theorem}\label{mainthgen}  Let $\underline{S}$ be  an arbitrary subset of a Dedekind domain $D$ and $f = \tfrac{g}{d} \in \mathrm{Int}(\underline{S},D  ) $ be  a polynomial  of type  $(\m,k).$ If $\underline{a}_0, \underline{a}_1, \ldots, \underline{a}_{\lmk-1}$ is a $d_{\m}$-sequence, then $f$ is irreducible in $\mathrm{Int}(\underline{S},D  ) $ iff the following holds:

for any factorization $g=g_1g_2$ in $D[x]$    and    a  prime ideal $P$  dividing $d$ such that $e_k$ is the maximum integer satisfying  $P^{e_k} \mid g_1(\underline{a}_i)\ \forall\ 0 \leq i \leq l(g_1) , $   there exists an integer $j$  satisfying  $ 0 \leq j \leq l(g_2)  $ and    ${w_{P}(\tfrac{(d)}{P^{e_k}})}   \nmid  g_2(\underline{a}_j).$
  
  \end{theorem}

\begin{proof} The steps of the proof are similar to that of   Theorem \ref{mainth}.  Suppose that   for   every factorization   $g=g_1g_2   $  in $D[x]$,    
 there exists a prime ideal 
 $P \mid d  $  such that  $P^{e_k} \mid g_1(\underline{a}_i)\ \forall\ 0 \leq i \leq l(g_1)  $   and  ${w_{P}(\tfrac{(d)}{P^{e_k}})}   \nmid  g_2(\underline{a}_j)$ for some  $ 0 \leq j \leq l(g_2).$
  Let us assume that the  polynomial $f$  is reducible in $\mathrm{Int}(\underline{S},D  ) $ . 
 Then there exist $d_1$ and $d_2$ such that 
 
 $$f= \dfrac{g_1}{d_1}  \dfrac{g_2}{d_2} , $$
 
where both of the polynomials, $\tfrac{ g_1 }{d_1}$ and $  \tfrac{g_2}{d_2} $  
are  integer-valued.
 
 Since  $\tfrac{ g_1 }{d_1}$ and  $\tfrac{ g_2 }{d_2}$  are integer-valued, hence for every $\pi \mid d$ there exists an 
 integer $c_k$   such that 
 $w_{\pi}(d_1)=  \pi^{c_k} \mid g_1(\underline{a}_i)\ \forall\ 0 \leq i \leq l(g_1)  $ and    $w_{\pi}(\tfrac{d}{\pi^{c_k}}) \mid g_2(\underline{a}_j)\ \forall\ 0 \leq j \leq l(g_2) .$ This means there exist a prime ideal 
 $P\mid d$ such that $   P^{e_k} \mid g_1(\underline{a}_i)\ \forall\ 0 \leq i \leq l(g_1)  $ and    $w_{P}(\tfrac{(d)}{P{e_k}}) \mid g_2(\underline{a}_j)\ \forall\ 0 \leq j \leq l(g_2) ,$ which is a contradiction to the assumption. The remaining part also follows by a similar way.

\end{proof}

 Throughout the proof of Theorem \ref{mainth} (or  Theorem \ref{mainthgen}) we used only unique factorization of the element $d$. Hence, Theorem   \ref{mainth} remains valid for all the domains where $d$ has a unique factorization into irreducibles (or into prime ideals) and $R_{(\pi)}$ is local ring for all $\pi$ dividing $d$. Therefore, sometimes our approach can be helpful in testing the irreducibility of an integer-valued polynomial over any subset of a domain.

  \medskip

To conclude,   this article is an initial step to test the irreducibility of   multivariate  integer-valued polynomials over   any subset  of a domain. This is still  ongoing work.  We believe that the   irreducibility of  multivariate integer-valued polynomials is more interesting than the   irreducibility of  the integer-valued polynomials in one variable. This is  a promising area of research that has not been explored so far. 
This article is the first step in this broad area of research and may provide some impetus to the readers to work in this exciting area.

    \section*{Acknowledgment}  We   thank the anonymous referee for his valuable  suggestions. The author also wishes to thank Professor S\'andor Kov\'acs (Editor, 
Periodica Mathematica Hungarica) for his help and guidance.   


\end{document}